\documentclass{compositio}
\usepackage[active]{srcltx}
\usepackage{amsmath}\usepackage{amssymb}\usepackage{amscd}\usepackage{amsthm}\usepackage{amsfonts}
\usepackage{graphicx}
\usepackage{cite}

\newtheorem{theorem}{\textbf{Theorem}}[section]
\newtheorem{corollary}[theorem]{\textbf{Corollary}}

\newtheorem{proposition}[theorem]{\textbf{Proposition}}
\newtheorem*{theo-intro}{\textbf{Theorem}}

\theoremstyle{definition}
\newtheorem{definition}[theorem]{\textbf{Definition}}

\theoremstyle{remark}
\newtheorem{remark}[theorem]{\textbf{Remark}}
\newtheorem{remarks}[theorem]{\textbf{Remarks}}
\newtheorem{example}[theorem]{\textbf{Example}}

\newcounter{pic}

\numberwithin{equation}{section}

\DeclareRobustCommand\mmodels{\Relbar\joinrel\mathrel{|}}
\def\C{\mathbb{C}}
\def\Z{\mathbb{Z}}

\def\hY{\widehat{Y}}
\def\hH{\widehat{H}}
\def\og{\overline{g}}
\def\oX{\overline{X}}
\def\tX{\widetilde{X}}
\newcommand{\Comp}{\operatorname{Comp}}
\newcommand{\Soc}{\operatorname{Soc}} 
\newcommand{\Irr}{\operatorname{Irr}} 
\newcommand{\Mat}{\operatorname{Mat}}
\newcommand{\bun}{\boldsymbol{1}}
\newcommand{\bv}{\boldsymbol{v}}
\newcommand{\btau}{\boldsymbol{\tau}}
\def\hB{B^{\text{aff}}}
\newcommand{\bx}{\boldsymbol{x}}
\def\trho{\widetilde{\rho}}
\def\tP{\widetilde{P}}

\begin{document}

\title{Invariants for links from classical and affine Yokonuma--Hecke algebras}

\author{L. Poulain d'Andecy}
\email{loic.poulain-dandecy@univ-reims.fr}
\address{Universit\'e de Reims Champagne-Ardenne, UFR Sciences exactes et naturelles, Laboratoire de Math\'ematiques EA 4535 Moulin de la Housse BP 1039, 51100 Reims, France}

\classification{57M25, 57M27, 20C08, 20F36}


\date{January, 2016}

\thanks{It is a pleasure to thank the organisers of the Thales workshop held in Athens in July 2015 and especially Sofia Lambropoulou for her interest in this work.}

\begin{abstract}
We present a construction of invariants for links using an isomorphism theorem for affine Yokonuma--Hecke algebras. The isomorphism relates affine Yokonuma--Hecke algebras with usual affine Hecke algebras. We use it to construct a large class of Markov traces on affine Yokonuma--Hecke algebras, and in turn, to produce invariants for links in the solid torus. By restriction, this construction contains the construction of invariants for classical links from classical Yokonuma--Hecke algebras. In general, the obtained invariants form an infinite family of 3-variables polynomials. As a consequence of the construction via the isomorphism, we reduce the number of invariants to study, given the number of connected components of a link. In particular, if the link is a classical link with $N$ components, we show that $N$ invariants generate the whole family.
\end{abstract}

\maketitle

\section{Introduction}

\textbf{1.} The Yokonuma--Hecke algebras (of type GL), denoted $Y_{d,n}$, have been used by J. Juyumaya and S. Lambropoulou to construct invariants for various types of links, in the same spirit as the construction of the HOMFLYPT polynomial from usual Hecke algebras. We refer to \cite{ChJKL} and references therein. In particular, the algebras $Y_{d,n}$ provide invariants for classical links and the natural question was to decide if these invariants were equivalent, or not, to the HOMFLYPT polynomial. This study culminated in the recent discovery \cite{ChJKL} that these invariants are actually topologically stronger than the HOMFLYPT polynomial (\emph{i.e.} they distinguish more links).

In \cite{JaPo}, another approach to study invariants coming from Yokonuma--Hecke algebras was developed. The starting point was the fact that the algebra $Y_{d,n}$ is isomorphic to a direct sum of matrix algebras with coefficients in tensor products of usual Hecke algebras. This allowed an explicit construction of Markov traces on $\{Y_{d,n}\}_{n\geq1}$ from the known Markov trace on Hecke algebras (on Hecke algebras, there is a unique Markov trace up to normalisation, and it gives the HOMFLYPT polynomial). In addition to its usefulness for the construction of Markov traces, the approach via the isomorphism also helps to study the resulting invariants. Indeed some properties of the invariants follow quite immediately from a precise understanding of the isomorphism (see paragraph \textbf{4} below).

Independently of which approach is used, another ingredient was added in \cite{JaPo}: a third parameter in the invariants. While the first two parameters come from the algebra $Y_{d,n}$, this third parameter $\gamma$ has its origin in the framed braid group, and corresponds to a certain degree of freedom one has when going from the framed braid group to the algebra $Y_{d,n}$. More precisely, we can deform the standard surjective morphism from the framed braid group algebra to its quotient $Y_{d,n}$ into a family of morphisms (depending on $\gamma$) respecting the braid relations and the Markov conditions. Another way of interpreting the parameter $\gamma$ is that it modifies the quadratic relation satisfied by the generators of $Y_{d,n}$. Its existence explains (or is reflected in) the fact that different presentations for $Y_{d,n}$ were used before. Juyumaya--Lambropoulou invariants correspond to certain specialisations of this parameter $\gamma$, depending on the chosen presentation. So the parameter $\gamma$ unifies every possible choices and yields more general invariants. It is indicated in \cite[Remark 8.5]{ChJKL} that changing the presentation seems to give a non-equivalent topological invariant.

\vskip .2cm
\textbf{2.} In this paper, we consider the affine Yokonuma--Hecke algebras (of type GL), denoted $\hY_{d,n}$. They were introduced in \cite{ChPo} in connections with the representation theory and the Jucys--Murphy elements of the classical Yokonuma--Hecke algebras. Our main goal here is to generalise for $\hY_{d,n}$ the whole approach to link invariants via the isomorphism theorem. The invariants are in general for links in the solid torus. The classical links are naturally contained in the solid torus links and, restricted to them, the obtained invariants correspond to the invariants obtained in \cite{JaPo} from $Y_{d,n}$ (naturally seen as a subalgebra of $\hY_{d,n}$). Specialising the parameter $\gamma$, we identify the Juyumaya--Lambropoulou invariants among them. For those invariants, we emphasize that we recover some known results \cite{ChJKL} by a different method and furthermore obtain some new results already in this particular case.

We start with an isomorphism between the algebra $\hY_{d,n}$ and a direct sum of matrix algebras with coefficients in tensor products of affine Hecke algebras. As done in \cite{Cu}, the isomorphism can be proved repeating the same arguments as for $Y_{d,n}$ (see \cite{JaPo} where the proof for $Y_{d,n}$ is presented, as a particular case of a more general result by G. Lusztig \cite[\S 34]{Lu}). Here we sketch a short different proof for $\hY_{d,n}$ using the known result for $Y_{d,n}$. We also prove the analogous theorem for the cyclotomic quotients of $\hY_{d,n}$ (with Ariki--Koike algebras replacing affine Hecke algebras). Useful for concrete use, the formulas for the generators are simple and given explicitly.

Concerning links, S. Lambropoulou constructed invariants, analogues of the HOMFLYPT polynomial, for links in the solid torus from affine Hecke algebras \cite{La2}. Then, it was explained in  \cite{ChPo2} how to obtain invariants for those links from the algebras $\hY_{d,n}$, unifying the methods of J. Juyumaya and S. Lambropoulou for $Y_{d,n}$ and the construction of S. Lambropoulou for affine Hecke algebras. Due to the recent results of \cite{ChJKL}, it is expected that the invariants obtained from $\hY_{d,n}$ are stronger than the ones obtained from affine Hecke algebras.

Here we follow the alternative approach which uses the isomorphism to construct Markov traces on the family of algebras $\{\hY_{d,n}\}_{n\geq1}$. To sum up, the Markov traces are constructed and can be calculated with the following steps: for an element of $\hY_{d,n}$, apply first the isomorphic map to obtain an element of the direct sum of matrix algebras; then, for each matrix, apply the usual trace which results in an element of a tensor product of affine Hecke algebras; finally apply a tensor product of Markov traces on affine Hecke algebras. Our result consists in obtaining the compatibility conditions relating the Markov traces appearing in different matrix algebras so that the preceding procedure eventually results in a genuine Markov trace on $\{\hY_{d,n}\}_{n\geq1}$.

With the definition used here, for a given $d>0$, the set of Markov traces on $\{\hY_{d,n}\}_{n\geq1}$ forms a vector space. From the isomorphism, a set of distinguished Markov traces appears naturally, which spans the set of all Markov traces constructed here. Thus, our study of Markov traces (and of invariants) is reduced to the study of these ``basic" Markov traces (and of the corresponding ``basic"  invariants). It turns out that these basic Markov traces are indexed, for a given $d>0$, by the non-empty subsets $S\subset\{1,\dots,d\}$ together with a choice, denoted formally by $\btau$, of $|S|$ arbitrary Markov traces on affine Hecke algebras. We note that if we restrict to $Y_{d,n}$, the parameter $\btau$ disappears and the basic Markov traces on $\{Y_{d,n}\}_{n\geq1}$ are indexed, for a given $d>0$, only by the non-empty subsets $S\subset\{1,\dots,d\}$. This recovers a result of \cite{JaPo}.

\vskip .2cm
\textbf{3.} Throughout the paper, we intended to give in details the connections between the two approaches, so that one would be able to pass easily from one to the other. This will allow in particular to specialise and translate all our results on the invariants to Juyumaya--Lambropoulou invariants as well.

Roughly speaking, J. Juyumaya and S. Lambropoulou constructed invariants from $Y_{d,n}$ in two steps \cite{JuLa2}. The same approach was followed in \cite{ChPo2} for $\hY_{d,n}$. First a certain trace map, analogous to the Ocneanu trace and satisfying a certain positive Markov condition, was constructed. Then a rescaling procedure was implemented, in order to produce genuine invariants. The rescaling procedure amounts to two things: a renormalisation of the generators and a renormalisation, depending on $n$, of the trace. In the approach presented here, the first step is included from the beginning in a more general quadratic relation for the generators. The second step is already included in the definition of a Markov trace, namely that it is a family, on $n$, of trace maps satisfying the two Markov conditions. As a consequence, to obtain invariants here, one directly applies the Markov trace  and no rescaling procedure is needed.

For the comparison, our first task is to explain that Juyumaya--Lambropoulou approach is equivalent to considering certain Markov traces (with the definition used here) and to relate their variables with the parameters considered here. Then we need to identify these Markov traces in terms of the ones constructed via the isomorphism theorem. We obtain finally the explicit decomposition of 
these Markov traces in terms of the basic Markov traces indexed by $S\subset\{1,\dots,d\}$ and $\btau$ as above.

In particular, for $Y_{d,n}$, this results in an explicit formula for the Juyumaya--Lambropoulou invariants, as studied in \cite{ChJKL}, in terms of the basic invariants constructed here. We note that, in this case, the parameter $\btau$ is not present, and that Juyumaya--Lambropoulou invariants are also parametrised, for a given $d>0$, by non-empty subsets of $\{1,\dots,d\}$. Nevertheless, they do not coincide with the basic invariants and the comparison formula is not trivial (see Formulas (\ref{compare2}) in Section \ref{Sec-inv}). In general, for $\hY_{d,n}$, we obtain  the expression of the invariants constructed in \cite{ChPo2} in terms of the basic invariants constructed here.

Concerning the third parameter $\gamma$, we recall that it was not present in the previous approach. Actually, one need to specialise it to a certain value in our invariants to recover the Juyumaya--Lambropoulou invariants. The two different presentations of $Y_{d,n}$ that were used, as in \cite{ChJKL}, correspond to two different values of $\gamma$ that we give explicitly. Similarly, for $\hY_{d,n}$, the invariants constructed in \cite{ChPo2} correspond to a certain specialisation of $\gamma$.

\vskip .2cm
\textbf{4.} We conclude this introduction by describing the main properties obtained for the invariants. As explained before, they follow quite directly from a precise understanding of the isomorphism, and are expressed easily in terms of the basic invariants defined here. The main results are:
\begin{itemize}
\item for $d>0$ and a non-empty subset $S\subset\{1,\dots,d\}$, the corresponding invariants coincide with invariants corresponding to $d'=|S|$ and the full set $\{1,\dots,d'\}$. Therefore, we only have to consider the full sets $\{1,\dots,d\}$ for different $d>0$.
\item further, given a number $N$ of connected components of a link, the invariants corresponding to $\{1,\dots,d\}$ are zero if $d>N$. So, given $N$, we only have to consider $d=1,\dots,N$.
\end{itemize}
Moreover, with the comparison results explained in paragraph \textbf{3}, it is easy to deduce the similar properties for invariants obtained via Juyumaya--Lambropoulou approach. The first item remains true as it is. The second item results in an explicit formula expressing, if $d>N$, the invariants corresponding to $\{1,\dots,d\}$ in terms of the invariants corresponding to $\{1,\dots,d'\}$ with $d'\leq N$. Specialising $\gamma $ to the appropriate values and restricting to classical links, we recover with the first item a result of \cite{ChJKL}. The second item in this case was proved only for $N\leq2$ also in \cite{ChJKL}.

\section{Affine Yokonuma--Hecke algebras}\label{sec-def}

Let  $d,n \in\Z_{>0}$ and $u$ and $v$ be indeterminates. We work over the ring $\C[u^{\pm1},v]$.
The  properties of the affine Yokonuma--Hecke algebras recalled here can be found in \cite{ChPo2}.

\subsection{Definitions}

We use $\mathfrak{S}_n$ to denote the symmetric group on $n$ elements, and $s_i$ to denote the transposition $(i,i+1)$. The affine Yokonuma--Hecke algebra $\hY_{d,n}$ is generated by elements
$$g_1,\ldots,g_{n-1},X_1^{\pm1},t_1,\ldots,t_n,$$
subject to the following defining relations (\ref{def-aff1})--(\ref{def-aff3}):
\begin{equation}\label{def-aff1}
\begin{array}{rclcl}
g_ig_j & = & g_jg_i &\quad & \mbox{for $i,j=1,\ldots,n-1$ such that $\vert i-j\vert > 1$,}\\[0.1em]
g_ig_{i+1}g_i & = & g_{i+1}g_ig_{i+1} && \mbox{for $i=1,\ldots,n-2$,}\\[0.1em]
X_1\,g_1X_1g_1 & =  & g_1X_1g_1\,X_1  &&\\[0.1em]
X_1g_i & = & g_iX_1 && \mbox{for $i=2,\ldots,n-1$,}\\[0.1em]
\end{array}
\end{equation}
\begin{equation}\label{def-aff2}
\hspace{-1.2cm}\begin{array}{rclcl}
t_it_j & =  & t_jt_i &\qquad&  \mbox{for $i,j=1,\ldots,n$,}\\[0.1em]
g_it_j & = & t_{s_i(j)}g_i && \mbox{for $i=1,\ldots,n-1$ and $j=1,\ldots,n$,}\\[0.1em]
t_j^d   & =  &  1 && \mbox{for $j=1,\ldots,n$,}\\[0.2em]
X_1t_j & = & t_jX_1 && \mbox{for $j=1,\ldots,n$,}
\end{array}
\end{equation}
\begin{equation}\label{def-aff3}
\hspace{-3.2cm}
\begin{array}{rclcl}
g_i^2  & = & u^2 + v \, e_{i} \, g_i &\quad& \mbox{for $i=1,\ldots,n-1$,}
\end{array}
\end{equation}
where $e_i :=\displaystyle\frac{1}{d}\sum_{1\leq s\leq d}t_i^s t_{i+1}^{-s}$. The elements $e_i$ are idempotents and we have:
\begin{equation}
g_i^{-1} = u^{-2}g_i - u^{-2}v\, e_i  \qquad \mbox{for all $i=1,\ldots,n-1$}.
\end{equation}

Let $w \in \mathfrak{S}_n$ and let $w=s_{i_1}s_{i_2}\ldots s_{i_r}$ be a reduced expression for $w$.
Since the generators $g_i$ of $\hY_{d,n}$ satisfy the same braid relations as the generators $s_i$ of $\mathfrak{S}_n$, Matsumoto's lemma implies that the following element does not depend on the reduced expression of $w$:
\begin{equation}\label{def-gw}g_w:=g_{i_1}g_{i_2}\ldots g_{i_r}\ .
\end{equation} 

Elements $X_2,\ldots,X_n$ of $\hY_{d,n}$ are defined inductively by
\begin{equation}\label{rec-X}
X_{i+1}:=u^{-2}g_iX_ig_i\ \ \ \ \text{for $i=1,\ldots,n-1$.}
\end{equation}
The elements $X_1,\dots,X_n$ commute with each other. They also commute with the generators $t_1,\dots,t_n$ and they satisfy
$g_jX_i=X_ig_j$ if $i\neq j,j+1$.

For $\lambda=(\lambda_1,\dots,\lambda_n)\in\Z^n$, we set $X^{\lambda}:=X_1^{\lambda_1}\dots X_n^{\lambda_n}$.
The following set of elements forms a basis of $\hY_{d,n}$:
\begin{equation}\label{basis}
\{\ t_1^{a_1}\dots t_n^{a_n}X^{\lambda}g_w\ |\ a_1,\dots,a_n\in\{1,\dots,d\}\,,\ \lambda\in\Z^n\,,\ w\in\mathfrak{S}_n\ \}\ .
\end{equation}
This fact has the following consequences:

$\bullet$ Recall that the Yokonuma--Hecke algebra $Y_{d,n}$ is presented by generators $g_1,\dots,g_{n-1}$, $t_1,\dots,t_n$ and defining relations those in (\ref{def-aff1})--(\ref{def-aff3}) which do not involve the generator $X_1$. We have that $Y_{d,n}$ is isomorphic to the subalgebra of $\hY_{d,n}$ generated by $g_1,\dots,g_{n-1},t_1,\dots,t_n$ (hence the common names for the generators).

$\bullet$ In particular, the commutative subalgebra $\mathcal{T}_{d,n}:=\langle t_1,\ldots, t_n\rangle$ of $\hY_{d,n}$ generated by $t_1,\dots,t_n$ is isomorphic to the group algebra of $(\mathbb{Z}/d\mathbb{Z})^n$.

$\bullet$ By definition, the affine Hecke algebra (of type GL) is $\hH_n:=\hY_{1,n}$. We have, for any $d>0$, that the quotient of $\hY_{d,n}$ by the relations $t_j=1$, $j=1,\dots,n$, is isomorphic to $\hH_n$. We denote by $\overline{\,\cdot\,}$ the corresponding surjective morphism from $\hY_{d,n}$ to $\hH_n$, and the generators of $\hH_n$ are denoted $\og_1,\dots,\og_{n-1},\oX_1^{\pm1}$.

$\bullet$ The subalgebra of $\hH_n$ generated by $\og_1,\dots,\og_{n-1}$ is the usual Hecke algebra, denoted $H_n$. We also have $H_n=Y_{1,n}$.

\subsection{Compositions of $n$}

Let $\operatorname{Comp}_d (n)$ be the set of {\it $d$-compositions} of $n$, that is the set of $d$-tuples $\mu=(\mu_1,\ldots,\mu_d)\in\Z_{\geq0}^d$ such that $\sum_{1\leq a\leq d} \mu_a =n$. We denote $\mu\mmodels_d n$.

For $\mu\mmodels_d n$, the Young subgroup $\mathfrak{S}^{\mu}$ is the subgroup $\mathfrak{S}_{\mu_1}\times\dots\times\mathfrak{S}_{\mu_d}$ of  $\mathfrak{S}_{n}$, where $\mathfrak{S}_{\mu_1}$ acts on the letters $\{1,\dots,\mu_1\}$, $\mathfrak{S}_{\mu_2}$ acts on the letters $\{\mu_1+1,\dots,\mu_2\}$, and so on. The subgroup $\mathfrak{S}^{\mu}$ is generated by the transpositions $s_i$ with $i\in I_{\mu}:=\{1,\ldots,n-1\}\setminus \{\mu_1,\mu_1+\mu_2,\ldots, \mu_{1}+\ldots +\mu_{d-1}\}$.

We denote by $\hH^{\mu}$ the algebra $\hH_{\mu_1} \otimes \ldots \otimes \hH_{\mu_d}$ (by convention $\hH_0:=\C[u^{\pm 1},v]$). It is isomorphic to the subalgebra of $\hH_n$ generated by $\oX_1^{\pm1},\dots,\oX_n^{\pm1}$ and $\og_i$, with $i\in I_{\mu}$, and is a free submodule with basis $\{\oX^{\lambda}\og_w \ |\ \lambda\in\Z^n\,,\ w\in \mathfrak{S}^{\mu}\}$.

Similarly, we have a subalgebra $H^{\mu}\cong H_{\mu_1} \otimes \ldots \otimes H_{\mu_d}$ of the Hecke algebra $H_n$. It is naturally a subalgebra of $\hH^{\mu}$ (generated only by $\og_i$, with $i\in I_{\mu}$).

For $\mu\mmodels_d n$, let $m_{\mu}$ be the index of the Young subgroup $\mathfrak{S}^{\mu}$ in $\mathfrak{S}_n$, that is,
\begin{equation}\label{mmu}
m_{\mu}:=\displaystyle \frac{n!}{\mu_1!\mu_2!\dots\mu_d!}\ .
\end{equation}
We define the \emph{socle} $\overline{\mu}$ of a $d$-composition $\mu$ by
 \begin{equation}\label{def-base-mu}
\overline{\mu}_a=\left\{\begin{array}{ll}
1 & \text{if $\mu_a\geq1$,}\\[0.2em]
0 &  \text{if $\mu_a=0$,}
\end{array}\right.\ \ \ \ \ \ \text{for $a=1,\dots,d$.}
\end{equation}
The composition $\overline{\mu}$ belongs to $\Comp_d(N)$ where $N$ is the number of non-zero parts in $\mu$. We denote by $\Soc_d$ the set of all socles of $d$-compositions, or in other words, $\Soc_d$ is the set of $d$-compositions whose parts belong to $\{0,1\}$. We note that there is a one-to-one correspondence between the set $\Soc_d$ and the set of non-empty subsets of $\{1,\dots,d\}$, given by
 \begin{equation}\label{bij-Soc}
\{1,\dots,d\}\supset S\ \ \longleftrightarrow\ \  \mu^S\in\Soc_d\,,\ \ \text{where}\ \  \mu^S_a=\left\{\begin{array}{ll}
1 & \text{if $a\in S$,}\\[0.2em]
0 &  \text{if $a\notin S$.}
\end{array}\right.
\end{equation}

\subsection{characters of $\mathcal{T}_{d,n}$}
 
Let $\{\xi_1,\dots,\xi_d\}$ be the set of roots of unity of order $d$. A complex character $\chi$ of the group $(\mathbb{Z}/d\mathbb{Z})^n$ is characterised by the choice of $\chi (t_j)\in \{\xi_1,\ldots ,\xi_d  \}$ for each $j=1,\ldots,n$. We denote by $\Irr(\mathcal{T}_{d,n})$ the set of complex characters of $(\mathbb{Z}/d\mathbb{Z})^n$, extended to the subalgebra $\mathcal{T}_{d,n}=\langle t_1,\dots,t_n\rangle$ of $\hY_{d,n}$.

For each $\chi \in \Irr \left(\mathcal{T}_{d,n}\right)$, we denote by $E_{\chi}$ the primitive idempotent of 
 $\mathcal{T}_{d,n}$ associated to $\chi$. Then the set $\{ E_{\chi}  \ |\ \chi \in  \Irr (\mathcal{T}_{d,n})\}$ is a basis of $\mathcal{T}_{d,n}$. Therefore, from the basis (\ref{basis}) of $\hY_{d,n}$, we obtain the following other basis of $\hY_{d,n}$: 
\begin{equation}\label{E-basis}
\{ E_{\chi} X^{\lambda} g_w \ |\ \chi \in  \Irr (\mathcal{T}_{d,n})\,,\ \lambda\in\Z^n\,,\ w\in \mathfrak{S}_n\}\ .
\end{equation}

\paragraph{\textbf{Permutations $\pi_{\chi}$.}} Let $\chi\in\Irr (\mathcal{T}_{d,n})$. For $a\in\{1,\ldots,d\}$, we let $\mu_a$ be the number of elements  $j\in \{1,\ldots,n\}$ such that $\chi (t_j)=\xi_a$. Then the sequence $(\mu_1,\ldots,\mu_d)$ is a $d$-composition of $n$ which we denote by 
$\Comp(\chi)$.

For a given $\mu\mmodels_d n$, we consider a particular character $\chi_0^{\mu}\in\Irr(\mathcal{T}_{d,n})$ such that $\Comp(\chi_0^{\mu})=\mu$.  The character $\chi_0^{\mu}$ is defined by
\begin{equation}\label{chi0-mu}
\left\{\begin{array}{ccccccc}
\chi_0^{\mu} (t_1)&=&\ldots & =& \chi_0^{\mu} (t_{\mu_1})&=& \xi_1\ ,\\[0.2em]
\chi_0^{\mu} (t_{\mu_1+1})&=&\ldots & =& \chi_0^{\mu} (t_{\mu_1+\mu_2})&=& \xi_2\ ,\\
\vdots &\vdots &\vdots &\vdots &\vdots &\vdots &\vdots  \\
\chi_0^{\mu} (t_{\mu_1+\dots+\mu_{d-1}+1})&=&\ldots & =& \chi_0^{\mu} (t_{n})&=& \xi_d\ .\\
\end{array}\right.
\end{equation}
The symmetric group  $\mathfrak{S}_n$ acts on the set $\Irr (\mathcal{T}_{d,n})$ by the formula $w(\chi)\bigl(t_i\bigr)=\chi (t_{w^{-1} (i)})$. The stabilizer of $\chi_0^{\mu}$ under the action of $\mathfrak{S}_n$ is the Young subgroup $\mathfrak{S}^{\mu}$. In each left coset in $\mathfrak{S}_n/\mathfrak{S}^{\mu}$, there is a unique representative of minimal length. So, for any $\chi\in\Irr(\mathcal{T}_{d,n})$ such that $\Comp(\chi)=\mu$, we define a permutation $\pi_{\chi}\in\mathfrak{S}_n$ by requiring that $\pi_{\chi}$ is the element of minimal length such that:
\begin{equation}\label{def-pi}
\pi_{\chi}(\chi_0^{\mu})=\chi\ .
\end{equation}

\section{Isomorphism Theorems}

We present isomorphism theorems for the algebras $\hY_{d,n}$ and their cyclotomic quotients. We sketch a short proof, which uses the corresponding result for $Y_{d,n}$ (see \cite[Section 3.1]{JaPo}). We are still working over $\C[u^{\pm1},v]$.

\subsection{Isomorphism theorem for affine Yokonuma--Hecke algebras}

For $\mu\mmodels_d n$, we consider the algebra $\Mat_{m_{\mu}}(\hH^{\mu})$ of matrices of size $m_{\mu}$ with coefficients in $\hH^{\mu}$. We recall that $m_{\mu}$, given by (\ref{mmu}), is the number of characters $\chi\in\Irr(\mathcal{T}_{d,n})$ such that $\Comp(\chi)=\mu$. So we index the rows and columns of a matrix in $\Mat_{m_{\mu}}(\hH^{\mu})$ by such characters. Moreover, for two characters $\chi,\chi'$ such that $\Comp(\chi)=\Comp(\chi')=\mu$, we denote by $\bun_{\chi,\chi'}$ the  matrix in $\Mat_{m_{\mu}}(\hH^{\mu})$ with 1 in line $\chi$ and column $\chi'$, and 0 everywhere else.
\begin{theorem}\label{theo-iso}
The affine Yokonuma--Hecke algebra $\hY_{d,n}$ is isomorphic to $\bigoplus_{\mu\mmodels_d\,n}\Mat_{m_{\mu}}(\hH^{\mu})$, the isomorphism being given on the elements of the basis (\ref{E-basis}) by
\begin{equation}\label{iso-aff}
\Psi_{d,n}\ \ :\ \ E_{\chi}X^{\lambda}g_{w^{-1}}\ \longmapsto\ u^{\ell(w^{-1})-\ell(\pi_{\chi}^{-1}w^{-1}\pi_{w(\chi)})}\,\bun_{\chi,w(\chi)}\,\oX^{\pi^{-1}_{\chi}(\lambda)}\og_{\pi_{\chi}^{-1}w^{-1}\pi_{w(\chi)}}\ ,
\end{equation}
where $\chi\in\Irr(\mathcal{T}_{d,n})$, $\lambda\in\Z^n$ and $w\in\mathfrak{S}_n$ ($\ell$ is the length function on $\mathfrak{S}_n$).
\end{theorem}
\begin{proof}[Sketch of a proof] We start with explicit formulas for the images of the generators of $\hY_{d,n}$ given below in (\ref{form-t})--(\ref{form-g}), and we check that the images of the generators satisfy all the defining relations (\ref{def-aff1})--(\ref{def-aff3}) of $\hY_{d,n}$. For the relations not involving the generator $X_1$, this is already known from the isomorphism theorem for $Y_{d,n}$. We omit the remaining straightforward verifications.

Thus, Formulas (\ref{form-t})--(\ref{form-g}) induce a morphism of algebras, and we check that it coincides with $\Psi_{d,n}$ given by (\ref{iso-aff}). Again, for the images of the elements of the basis of the form $E_{\chi}g_{w^{-1}}$, this is already known from the $Y_{d,n}$ situation. The multiplication by $X^{\lambda}$ is straightforward.

It remains to check that $\Psi_{d,n}$ is bijective. The surjectivity follows from a direct inspection of Formula (\ref{form-X}) together with the already known fact that every $\Mat_{m_{\mu}}(H^{\mu})$ is in the image of $\Psi_{d,n}$. The injectivity can be checked directly. Indeed, assume that a certain linear combination $\sum_{\chi,\lambda,w}c_{\chi,\lambda,w}E_{\chi}X^{\lambda}g_{w^{-1}}$ is in the kernel of $\Psi_{d,n}$. Then for every $\chi,\lambda,w$, we obtain that
\[\sum_{w'}u^{\ell(w'^{-1})-\ell(\pi_{\chi}^{-1}w'^{-1}\pi_{w'(\chi)})}c_{\chi,\lambda,w'}\,\og_{\pi_{\chi}^{-1}w'^{-1}\pi_{w'(\chi)}}=0\ ,\]
where the sum is over $w'\in\mathfrak{S}_n$ such that $w'(\chi)=w(\chi)$. For $w',w''$ satisfying this condition, we have $\pi_{\chi}^{-1}w'^{-1}\pi_{w'(\chi)}=\pi_{\chi}^{-1}w''^{-1}\pi_{w''(\chi)}$ if and only if $w'=w''$, and therefore every coefficients in the above sum are $0$.
\end{proof}

\paragraph{\textbf{Formulas for the generators.}} Here, we give the images under the isomorphism $\Psi_{d,n}$ of the generators of $\hY_{d,n}$. We recall that $\sum E_{\chi}=1$ in $\hY_{d,n}$, the sum being over $\Irr(\mathcal{T}_{d,n})$. Let $\chi\in\Irr(\mathcal{T}_{d,n})$ and set $\mu=\Comp(\chi)$. Then, by definition of $\pi_{\chi}$, it is straightforward to see that  $\pi_{\chi}^{-1}(j)=\mu_1+\dots+\mu_{a-1}+\alpha$, where $\chi(t_j)=\xi_a$ and $\alpha=\sharp\{k\leq j\ |\ \chi(t_k)=\xi_a\}$.
\begin{itemize}
\item Let $j\in\{1,\dots,n\}$. We have:
\begin{equation}\label{form-t}
t_j=\sum_{\chi\in\text{Irr}(\mathcal{T}_{d,n})}E_{\chi}t_j=\sum_{\chi\in\text{Irr}(\mathcal{T}_{d,n})}E_{\chi}\chi(t_j)\ \longmapsto\ \sum_{\chi\in\text{Irr}(\mathcal{T}_{d,n})}\bun_{\chi,\chi}\,\chi(t_j)\ .
\end{equation}
It follows that the image of $e_i$, $i=1,\dots,n-1$, is a sum of diagonal matrices; the coefficient in position $\chi$ is 1 if $\chi(t_i)=\chi(t_{i+1})$, and 0 otherwise.
\item Let $j\in\{1,\dots,n\}$. We have:
\begin{equation}\label{form-X}
X_j=\sum_{\chi\in\text{Irr}(\mathcal{T}_{d,n})}E_{\chi}X_j\ \longmapsto\ \sum_{\chi\in\text{Irr}(\mathcal{T}_{d,n})}\bun_{\chi,\chi}\,\oX_{\pi^{-1}_{\chi}(j)}\ .
\end{equation}
\item Let $i\in\{1,\dots,n-1\}$. We have:
\begin{equation}\label{form-g}g_i=\sum_{\chi\in\text{Irr}(\mathcal{T}_{d,n})}E_{\chi}g_i\ \ \ \ \ \ \text{and}\ \ \ \ \ E_{\chi}g_i\longmapsto\left\{\begin{array}{ll}
u\,\bun_{\chi,s_i(\chi)} & \text{if $s_i(\chi)\neq\chi$,}\\[0.4em]
\bun_{\chi,\chi}\,\og_{\pi_{\chi}^{-1}(i)} & \text{if $s_i(\chi)=\chi$\ .}
\end{array}\right.
\end{equation}
The first line follows from $\pi_{s_i(\chi)}=s_i\pi_{\chi}$ if $s_i(\chi)\neq\chi$. The second line follows from $\pi^{-1}_{\chi}(i+1)=\pi^{-1}_{\chi}(i)+1$ if $s_i(\chi)=\chi$.
\end{itemize}

\subsection{Isomorphism Theorem for cyclotomic Yokonuma--Hecke algebras}

Let $\bv=(v_1,\dots,v_m)\subset\C[u^{\pm1},v]\backslash\{0\}$ be an $m$-tuple of non-zero parameters for a certain $m\in\Z_{>0}$ (equivalently, one could consider $v_1,\dots,v_m$ as indeterminates and work over the extended ring $\C[u^{\pm1},v,v_1^{\pm1},\dots,v_m^{\pm1}]$). The cyclotomic Yokonuma--Hecke algebra $Y_{d,n}(\bv)$ is the quotient of the affine Yokonuma--Hecke algebra $\hY_{d,n}$ by the relation
\begin{equation}\label{rel-cyc}
(X_1-v_1)\dots(X_1-v_m)=0\ .
\end{equation}
It is shown in \cite{ChPo2} that the algebra $Y_{d,n}(\bv)$ is a free $\C[u^{\pm1},v]$-module with basis
\[\{t_1^{a_1}\dots t_n^{a_n}X^{\lambda}g_w\ |\ a_1,\dots,a_n\in\{1,\dots,d\}\,,\ \lambda\in\{0,\dots,m-1\}^n\,,\ w\in\mathfrak{S}_n\}\ .\]
In particular, if $m=1$, $Y_{d,n}(\bv)$ is isomorphic to the Yokonuma--Hecke algebra $Y_{d,n}$.

Similarly, the cyclotomic Hecke algebra $H_n(\bv)$ (or the Ariki--Koike algebra) is the quotient of the affine Hecke algebra $\hH_{n}$ by the relation
$(\oX_1-v_1)\dots(\oX_1-v_m)=0$. Equivalently, it is the quotient of the cyclotomic Yokonuma--Hecke algebra $Y_{d,n}(\bv)$ by the relations $t_j=1$, $j=1,\dots,n$. It is a free $\C[u^{\pm1},v]$-module with basis $\{\oX^{\lambda}\og_w\ |\ \lambda\in\{0,\dots,m-1\}^n\,,\ w\in\mathfrak{S}_n\}$.

For $\mu\mmodels_d n$, we set $H(\bv)^{\mu}:=H_{\mu_1}(\bv)\otimes\dots\otimes H_{\mu_d}(\bv)$. By definition, $H(\bv)^{\mu}$ is the quotient of the algebra $\hH^{\mu}$ by the relations
\begin{equation}\label{rel-cyc-mu}
(\oX_{\mu_1+\dots+\mu_{a-1}+1}-v_1)\dots(\oX_{\mu_1+\dots+\mu_{a-1}+1}-v_m)=0\ ,\ \ \ \ \ a=1,\dots,d\ .
\end{equation}

\begin{corollary}
The cyclotomic Yokonuma--Hecke algebra $Y_{d,n}(\bv)$ is isomorphic to the direct sum $\bigoplus_{\mu\mmodels_d
\,n}\Mat_{m_{\mu}}\bigl(H(\bv)^{\mu}\bigr)$\,.
\end{corollary}
\begin{proof} Let $I_{\bv}$ be the (two-sided) ideal of $\hY_{d,n}$ generated by the left hand side of the relation (\ref{rel-cyc}). For $\mu\mmodels_d n$, let $\overline{I}^{\mu}_{\bv}$ be the ideal of $\hH^{\mu}$ generated by the left hand sides of the relations (\ref{rel-cyc-mu}). The corollary follows from Theorem \ref{theo-iso} together with the fact that $\Psi_{d,n}(I_{\bv})=\bigoplus_{\mu\mmodels_d\,n}\Mat_{m_{\mu}}\bigl(\overline{I}^{\mu}_{\bv}\bigr)$. It remains to check this fact.

The inclusion ``$\subset$" follows at once from Formula (\ref{form-X}) for $j=1$. For the other inclusion, let $\mu\mmodels_d n$. Let $a\in\{1,\dots,d\}$ such that $\mu_a\neq0$, so that there is a character $\chi$ with $\Comp(\chi)=\mu$ and $\chi(t_1)=\xi_a$. Again, Formula (\ref{form-X}) for $j=1$ gives $\Psi_{d,n}(E_{\chi}X_1)=\bun_{\chi,\chi}\oX_{\mu_1+\dots+\mu_{a-1}+1}$. Therefore, for every generators of $\overline{I}^{\mu}_{\bv}$, we have in $\Psi_{d,n}(I_{\bv})$ a matrix in $\Mat_{m_{\mu}}\bigl(\hH^{\mu}\bigr)$ with the generator as one diagonal element and $0$ everywhere else. As $\Psi_{d,n}(I_{\bv})$ is an ideal, this shows that $\Mat_{m_{\mu}}\bigl(\overline{I}^{\mu}_{\bv}\bigr)$ is included in $\Psi_{d,n}(I_{\bv})$.
\end{proof}

\section{Markov traces on affine Yokonuma--Hecke algebras}\label{Sec-Mark}
From now on, we extend the ground ring $\C[u^{\pm1},v]$ to $\C[u^{\pm1},v^{\pm1}]$, and we consider our algebras over this extended ring.

\subsection{Definition of Markov traces on $\{\hY_{d,n}\}_{n\geq1}$ and $\{\hH_n\}_{n\geq1}$}

A Markov trace on the family of algebras $\{\hY_{d,n}\}_{n\geq1}$ is a family of linear functions $\{\rho_{d,n}\ :\ \hY_{d,n}\to \mathbb{C}[u^{\pm1},v^{\pm1}]\}_{n\geq1}$ satisfying:
\begin{equation}\label{Markov-rho}
\begin{array}{ll}
\rho_{d,n}(xy)=\rho_{d,n}(yx)\ ,& \text{$n\geq1$ and $x,y\in \hY_{d,n}$;}\\[0.2em]
\rho_{d,n+1}(xg_n)=\rho_{d,n+1}(xg_n^{-1})=\rho_{d,n}(x)\ ,\quad & \text{$n\geq1$ and $x\in \hY_{d,n}$.}
\end{array}
\end{equation}
A Markov trace on the family of algebras $\{\hH_{n}\}_{n\geq1}$ is a family of linear functions $\{\tau_n\ :\ \hH_{n}\to \mathbb{C}[u^{\pm1},v^{\pm1}]\}_{n\geq1}$ satisfying:
\begin{equation}\label{Markov-tau}
\begin{array}{ll}
\tau_n(xy)=\tau_n(yx)\ ,& \text{$n\geq1$ and $x,y\in \hH_{n}$;}\\[0.2em]
\tau_{n+1}(x\og_n)=\tau_{n+1}(x\og_n^{-1})=\tau_n(x)\ ,\quad & \text{$n\geq1$ and $x\in \hH_{n}$.}
\end{array}
\end{equation}

Recall the definition of $\Soc_d$ from Section \ref{sec-def}. For each $\overline{\mu}\in\Soc_d$ and each $a\in\{1,\dots,d\}$ such that $\overline{\mu}_a\neq0$, we choose a Markov trace $\{\tau^{\overline{\mu},a}_n\}_{n\geq1}$ on $\{\hH_n\}_{n\geq1}$. By convention, $\hH_0:=\C[u^{\pm1},v^{\pm1}]$ and maps of the form $\tau_0^{\overline{\mu},a}$ are identities on $\hH_0$. Below in (\ref{rho-n}), each term in the sum over $\mu\mmodels_d n$ acts on $\Mat_{m_{\mu}}(\hH^{\mu})$. We skip the proof of the following theorem. It can be done exactly as in \cite[Lemma 5.4]{JaPo}.
\begin{theorem}\label{theo-mark}
The following maps form a Markov trace on $\{\hY_{d,n}\}_{n\geq1}$:
\begin{equation}\label{rho-n}
\Bigl(\sum_{\mu\mmodels_d n}(\tau^{\overline{\mu},1}_{\mu_1}\otimes\dots\otimes \tau^{\overline{\mu},d}_{\mu_d})\circ\operatorname{Tr}_{\Mat_{m_{\mu}}} \Bigr)\circ\Psi_{d,n}\ ,\ \ \ \ \ \ \ \ n\geq 1\,.
\end{equation}
\end{theorem}
Roughly speaking, to construct a Markov trace on $\hY_{d,n}$, after having applied the isomorphism $\Psi_{d,n}$ and the usual trace of a matrix, we must choose and apply a Markov trace on each component of $\hH^{\mu}=\hH_{\mu_1}\otimes\dots\otimes\hH_{\mu_d}$ for each $\mu$. This choice of Markov traces is restricted: if $\mu$ and $\mu'$ have the same number of non-zero components, the chosen Markov traces must coincide; otherwise, they can be chosen independently.

\paragraph{\textbf{Basic Markov traces.}} Recall the bijection (\ref{bij-Soc}) between $\Soc_d$ and non-empty subsets of $\{1\dots,d\}$. Following the theorem, we define some distinguished Markov traces as follows:
\begin{itemize}
\item Choose a non-empty $S\subset\{1,\dots,d\}$ and consider the associated $\mu^S\in\Soc_d$. Choose a Markov trace $\{\tau^{a}_n\}_{n\geq1}$ on $\{\hH_n\}_{n\geq1}$ for each $a\in S$, and set $\tau^{\mu^S\!,\,a}_n=\tau^{a}_n$, $n\geq1$, in (\ref{rho-n}).
\item Then, in (\ref{rho-n}), set all other Markov traces $\{\tau^{\overline{\mu},a}_n\}_{n\geq1}$ with $\overline{\mu}\neq\mu^S$ to be $0$.
\end{itemize}
We denote formally the choice of Markov traces in the first item by $\btau$ and denote by $\{\rho_{d,n}^{S,\btau}\}_{n\geq1}$ the resulting Markov trace on $\{\hY_{d,n}\}_{n\geq1}$. We call it a \emph{basic} Markov trace. Every Markov traces constructed in the preceding theorem is a linear combination of basic Markov traces $\{\rho_{d,n}^{S,\btau}\}_{n\geq1}$, where $S$ and $\btau$ vary.

\begin{example}
A map on $\hY_{d,n}$ can be seen, up to $\Psi_{d,n}$, as acting on the direct sum of matrix algebras. This way, for a given $S$, the maps $\rho_{d,n}^{S,\btau}$ are non-zero only on the summands $\Mat_{m_{\mu}}(\hH^{\mu})$ such that $\overline{\mu}=\mu^S$, that is, such that $\mu_a\neq 0$ if and only if $a\in S$. As examples:
\begin{itemize}
\item if $S=\{k\}$ then $\rho_{d,n}^{S,\btau}$ is non-zero only on $\Mat_{m_{\mu}}(\hH^{\mu})$, for $\mu=(0,\dots,0,n,0,\dots,0)$ with $n$ in position $k$. In this case, $\hH^{\mu}=\hH_n$;
\item we will see that it is enough to consider the situation $S=\{1,\dots,d\}$. In this case, $\rho_{d,n}^{S,\btau}$ is non-zero only on $\Mat_{m_{\mu}}(\hH^{\mu})$, for $\mu$ with all parts different from 0.
\end{itemize}
\end{example}

\begin{remarks}\label{rem-mark}
\textbf{(i)} By restriction to the subalgebra $Y_{d,n}$ of $\hY_{d,n}$, a Markov trace on $\{\hY_{d,n}\}_{n\geq1}$ reduces to a Markov trace on $\{Y_{d,n}\}_{n\geq1}$ (and similarly for $\hH_n$ and $H_n$). On $\{H_n\}_{n\geq1}$, there is a unique Markov trace up to a normalisation factor. Therefore, the choice of the Markov traces $\tau^{\overline{\mu},a}_n$ in the theorem above reduces, for $Y_{d,n}$, to a choice of an overall factor $\alpha_{\overline{\mu}}$ for each $\overline{\mu}\in\Soc_d$. This is the result proved in \cite{JaPo}. In other words, for $Y_{d,n}$, the basic Markov traces are parametrised by $\Soc_d$, or equivalently, by the non-empty subsets of $\{1,\dots,d\}$.

\textbf{(ii)} Let $\mathfrak{Mark}(\hH_n)$ be the space of Markov traces on $\{\hH_{n}\}_{n\geq 1}$. The space spanned by the basic Markov traces $\rho_{d,n}^{S,\btau}$ is isomorphic to
\[\bigoplus_{1\leq k \leq d}\left(\begin{array}{c} d \\ k \end{array}\right) \mathfrak{Mark}(\hH_n)^{\otimes k}\ .\]
If we restrict to a subspace of $\mathfrak{Mark}(\hH_n)$ of dimension $D$, we obtain a space of Markov traces on $\{\hY_{d,n}\}_{n\geq 1}$ of dimension $(D+1)^d-1$. In particular, for $Y_{d,n}$, the dimension is $2^d-1$. We note that a full description of the space $\mathfrak{Mark}(\hH_n)$ does not seem to be known (and similarly for the cyclotomic quotients $H_n(\bv)$ others than $H_n$).
\end{remarks}

\section{Invariants for links}\label{Sec-inv}

Let $\gamma$ be another indeterminate. We work from now on over the ring $R:=\C[u^{\pm1},v^{\pm1},\gamma^{\pm1}]$ and we consider now all algebras over this extended ring $R$.

We sketch a construction of invariants with values in $R$ for $\Z/d\Z$-framed solid torus links. We refer to \cite{La2,ChPo2} for definitions and fundamental results (as the analogues of Alexander and Markov theorems) concerning solid torus links and their $\Z/d\Z$-framed versions. Note that any invariant for $\Z/d\Z$-framed links is also an invariant of non-framed links, simply by considering links with all framings equal to 0.

The set of classical links is naturally included in the set of solid torus links (in other words, the braid group is naturally a subgroup of the affine braid group). The construction here includes, by restriction to the subalgebras $Y_{d,n}$, the construction for $\Z/d\Z$-framed classical links explained in \cite[Section 6]{JaPo}. As the construction and the results equally apply to the classical and the solid torus situations, we will simply use the word \emph{link} to refer to both types of links.

\subsection{Definition of the invariants}

As in \cite{ChPo2}, we denote by $\hB_n$ the affine braid group on $n$ strands, and by $\Z/d\Z\wr \hB_n$ the $\Z/d\Z$-framed affine braid group. The generators of $\Z/d\Z\wr \hB_n$ are denoted $\sigma_1,\ldots,\sigma_{n-1},\sigma_0,t_1,\ldots,t_n$. The defining relations are (\ref{def-aff1})-(\ref{def-aff2}) with $g_i$ replaced by $\sigma_i$ and $X_1$ by $\sigma_0$. The algebra $\hY_{d,n}$ is thus a quotient of the group algebra of $\Z/d\Z\wr \hB_n$ by the relation (\ref{def-aff3}).

The subgroup of $\Z/d\Z\wr \hB_n$ generated by $\sigma_1,\ldots,\sigma_{n-1},\sigma_0$ is $\hB_n$. The algebra $\hH_{n}$ is a quotient of the group algebra of $\hB_n$ by the relation $\sigma_i^2=u^2+v\sigma_i$, $i=1,...,n-1$. Finally, the subgroup  $\sigma_1,\ldots,\sigma_{n-1}$ of $\hB_n$ is the classical braid group.

\paragraph{\textbf{Invariants $P^{\tau}_L(u,v)$ from $\hH_{n}$.}} Let $\{\tau_n\}_{n\geq1}$ be a Markov trace on $\{\hH_n\}_{n\geq1}$. From the Alexander and Markov theorems for non-framed links (see \cite{La2}), we construct the invariant $P_L^{\tau}(u,v)$, for a link $L$, as follows:
\[L\longmapsto\ \beta_L\in \hB_n\ \longmapsto\,\pi_n(\beta_L)\,\longmapsto \,\tau_{n}\bigl(\pi_{n}(\beta_L)\bigr)=:P_L^{\tau}(u,v)\in\C[u^{\pm1},v^{\pm1}]\ ,\]
where $\beta_L$ is a braid closing to $L$ and $\pi_n$ is the natural morphism from $R\hB_n$ to $\hH_n$, given on the generators by $\sigma_i\mapsto \og_i$, $i=1,...,n-1$ and $\sigma_0\mapsto \oX_1$. 

From the fact that there is a unique Markov trace (up to normalisation) on the usual Hecke algebras $\{H_n\}_{\geq1}$, all the invariants $P_L^{\tau}$, restricted to the set of classical links, reduce (up to normalisation) to the unique invariant coming from $H_n$, the HOMFLYPT polynomial.

\paragraph{\textbf{Invariants $P_L^{d,S,\btau}(u,v,\gamma)$ from $\hY_{d,n}$.}}
We consider the following map from $R[\Z/d\Z\wr \hB_n]$ to $\hY_{d,n}$ given on the generators by:
\[\delta_{d,n}\ :\ \ \ t_j\mapsto t_j\,,\ \ \ \ \ \  \sigma_0\mapsto X_1\,,\ \ \ \ \ \ \ \sigma_i\mapsto \bigl(\gamma+(1-\gamma)e_i\bigr)g_i\ .\]
One proves as in \cite[Section 6]{JaPo} that, first, $\delta_{d,n}$ extends to a morphism of algebras, and moreover, that the following procedure defines invariants for $\Z/d\Z$-framed links.
\begin{definition} Let $\{\rho_{d,n}\}_{n\geq1}$ be a Markov trace on $\{\hY_{d,n}\}_{n\geq1}$. For a $\Z/d\Z$-framed link $L$, the invariant $P_L^{\rho_d}(u,v,\gamma)$ is defined as follows
\[L\longmapsto\ \beta_L\in \Z/d\Z\wr \hB_n\ \longmapsto\,\delta_{d,n}(\beta_L)\,\longmapsto \,\rho_{d,n}\bigl(\delta_{d,n}(\beta_L)\bigr)=:P_L^{\rho_d}(u,v,\gamma)\in R\ ,\]
where $\beta_L$ is a $\Z/d\Z$-framed braid closing to $L$.
\end{definition}
From the preceding section, it is enough to consider the basic Markov traces $\{\rho^{S,\btau}_{d,n}\}_{n\geq1}$. We denote by $P_L^{d,S,\btau}$ the corresponding invariant and refer to it as a \emph{basic} invariant.

If we restrict a basic invariant to the set of classical $\Z/d\Z$-framed links, it does not depend on $\btau$, and coincides with the invariants constructed in \cite{JaPo}.

\begin{remark}\label{quadr-gamma} In the definition of the maps $\delta_{d,n}$, a rule $\sigma_i\mapsto (\alpha +\beta e_i)g_i$ would be enough to give a morphism of algebras. The condition $\alpha+\beta=1$ is necessary for the construction of invariants. We note that, considering the map $\delta_{d,n}$ is equivalent to changing the quadratic relations $g_i^2=u^2+ve_ig_i$ to $g_i^2= u^2\gamma^2+u^2(1-\gamma^2)e_i+ve_ig_i$. The role of $\gamma$ is therefore to interpolate between different presentations of $\hY_{d,n}$.
\end{remark}

\subsection{Comparison with other approaches}

\paragraph{\textbf{For non-framed links from $\hH_n$ \cite{La2}.}} Define $\tX_i\in\hY_{d,n}$, $i=1,\dots,n$, by the following formulas:
\[\tX_1:=X_1\ \ \ \ \ \ \text{and}\ \ \ \ \ \ \tX_{i+1}:=g_i^{-1}\tX_i g_i\,,\ \ i=1,\dots,n-1\,.\]
Similarly, we have the images $\overline{\tX}_1,\dots,\overline{\tX}_n$ of $\tX_1,\dots,\tX_n$ in $\hH_n$. 

Let $\bx:=\{x_a\}_{a\in\Z}\subset\C[u^{\pm1},v^{\pm1}]$ be a set of parameters with $x_0:=1$. From the results in \cite{La2} (see Remark \ref{rem-compare0} below), we have a unique Markov trace on $\{\hH_n\}_{n\geq1}$, which satisfies in addition
\[\tau^{\bx}_n\bigl(\overline{\tX}^a_n\,h\bigr)=x_a\tau^{\bx}_n(h)\,,\ \ \ \ \ \ \forall n\geq1\,,\ \ \forall a\in\Z\,,\ \ \forall h\in\hH_{n-1}\ .\]
The corresponding invariant of non-framed links is denoted $P^{\bx}_L(u,v)$.

\begin{remark}\label{rem-compare0} In \cite{La2}, the quadratic relation of $\hH_n$ was $\og_i^2=q+(q-1)\og_i$ and a certain trace $\text{tr}$, depending on another parameter $z$ was constructed. Setting $\lambda:=\frac{z+1-q}{qz}$, an invariant $\mathcal{X}_L(\sqrt{q},\sqrt{\lambda})$ was obtained, by rescaling the generators, $\og_i\leadsto \sqrt{\lambda}\og_i$, and then rescaling the trace. This is equivalent, in our approach, to setting 
$$u=\sqrt{\lambda q}\,,\ \ \ v=\sqrt{\lambda}(q-1)\ \ \ \ \ \text{and}\ \ \ \ \ \tau_n^{\bx}:=(\sqrt{\lambda}z)^{1-n}\text{tr}=\bigl(v^{-1}(1-u^2)\bigr)^{n-1}\text{tr}\,,\ \ \forall n\geq1\,.$$
In conclusion, we have $\mathcal{X}_L(\sqrt{q},\sqrt{\lambda})=P^{\bx}_L(u,v)$.
\end{remark}

\paragraph{\textbf{For $\Z/d\Z$-framed links from $\hY_{d,n}$ \cite{ChPo2,ChJKL,JuLa2}.}} Recall that $\{\xi_1,\dots,\xi_d\}$ is the set of $d$-th roots of unity. Fix a non-empty subset $D\subset\{1,\dots,d\}$ and, for each $k\in D$, a set $\bx^{(k)}:=\{x^{(k)}_a\}_{a\in\Z}$ of parameters in $\C[u^{\pm1},v^{\pm1}]$ with $x^{(k)}_0=1$. Denote formally $\bx$ the set $\bx^{(k)}$ with $k\in D$.

From the results of \cite{ChPo2} (and \cite{JuLa2} in the non-affine case), we have a unique Markov trace, denoted $\{\trho_{d,n}^{D,\bx}\}_{n\geq1}$, on $\{\hY_{d,n}\}_{n\geq1}$ which satisfies in addition, for all $n\geq1$,
\begin{equation}\label{cond-x}
\trho_{d,n}^{D,\bx}\bigl(\tX^a_nt_n^b\,h\bigr)=x_{a,b}\,\trho_{d,n}^{D,\bx}(h)\,,\ \ \ \ \ \ \forall a\in\Z\,,\ \ \forall b\in\{1,\dots,d\}\,,\ \ \forall h\in\hY_{d,n-1}\ ,
\end{equation}
where the parameters $x_{a,b}$ are given by 
\begin{equation}\label{para-x}
x_{a,b}=\frac{1}{\vert D\vert}\sum_{k\in D}x^{(k)}_a\xi_k^b\,,\ \ \ \ \ \text{for all $a\in\Z$ and $b\in\{1,\dots,d\}$.}
\end{equation}
The corresponding invariant of $\Z/d\Z$-framed links, we denote $\tP^{d,D,\bx}_L(u,v,\gamma)$. By restriction to classical $\Z/d\Z$-framed links, the parameters $\bx$ do not appear, and we obtain invariants labelled by $d$ and $D$, denoted $\tP^{d,D}_L(u,v,\gamma)$ (see Remarks below).

\begin{remarks}\label{rem-compare1} \textbf{(i)} In \cite{ChPo2}, the quadratic relation of $\hY_{d,n}$ was $g_i^2=1+(q-q^{-1})e_ig_i$ and a certain trace $\text{tr}$, depending on another parameter $z$ was constructed. An invariant $\Phi^{d,D,\bx}_L(q,z)$ was obtained. With $\lambda_D:=\frac{\vert D\vert z-(q-q^{-1})}{\vert D\vert z}$, this is equivalent, in our approach, to setting 
$$u=\sqrt{\lambda_D}\,,\ \ v=\sqrt{\lambda_D}(q-q^{-1})\,,\ \ \gamma=1\ \ \ \ \text{and}\ \ \ \ \trho_{d,n}^{D,\bx}:=\bigl(\vert D\vert v^{-1}(1-u^2)\bigr)^{n-1}\text{tr}\,,\ \  \forall n\geq1\,.$$
In conclusion, we have $\Phi^{d,D,\bx}_L(q,z)=\tP^{d,D,\bx}_L(u,v,1)\,$.
Note that the construction in \cite{ChPo2} corresponds to the particular case $\gamma=1$ here.

\textbf{(ii)} The restriction of the above procedure to the subalgebra $Y_{d,n}$ (the non-affine case) gives the comparison of invariants constructed here with the invariants studied in \cite{ChJKL}. The invariants only depend on $d$ and $D$ (not on $\bx$), and were denoted in \cite{ChJKL} by $\Phi_{d,D}(q,z)$.

\textbf{(iii)} Originally, in \cite{JuLa2}, invariants (in the non-affine case) were constructed using $Y_{d,n}$ with a different quadratic equation, namely with $g_i^2=1+(q^2-1)e_i+(q^2-1)e_ig_i$. A certain trace $\tilde{\text{tr}}$, depending on a parameter $\tilde{z}=q z$ was constructed and invariants, denoted $\Gamma_{d,S}$ in \cite{ChJKL}, were obtained. With the same $\lambda_D$ as above, the rescaling of the generators is now $g_i\leadsto \sqrt{\lambda_D}q^{-1} g_i$. A short calculation and comparison with the formula in Remark \ref{quadr-gamma} shows that $u$ and $v$ are as in item \textbf{(i)}, while now $\gamma=q^{-1}$. As a conclusion, we have, for any classical $\Z/d\Z$-framed link $L$,
\begin{equation}\label{compare}
\Phi_{d,D}(q,z)(L)=\tP^{d,D}_L(u,v,1)\ \ \ \ \ \text{and}\ \ \ \ \ \ \Gamma_{d,D}(q,z)(L)=\tP^{d,D}_L(u,v,q^{-1})\ .
\end{equation}
\end{remarks}

\subsection{Comparison in terms of the basic invariants $P^{d,S,\btau}_L(u,v,\gamma)$}

 We fix as above $D$ and $\bx$. It remains to identify the Markov trace $\{\trho_{d,n}^{D,\bx}\}_{n\geq1}$ in terms of the basic Markov traces $\{\rho_{d,n}^{S,\btau}\}_{n\geq1}$ constructed after Theorem \ref{theo-mark}.
\begin{proposition} For each $S\subseteq D$, let $\btau$ be obtained by taking, for each $k\in S$, the Markov trace $\{\tau^{\bx^{(k)}}\}_{n\geq1}$ in position $k$. Then we have
\begin{equation}\label{compare1}
\{\trho_{d,n}^{D,\bx}\}_{n\geq1}=\frac{1}{\vert D\vert}\sum_{S\subseteq D}\bigl(v^{-1}(1-u^2)\bigr)^{\vert S\vert -1}\,\{\rho_{d,n}^{S,\btau}\}_{n\geq1}\,.
\end{equation}
\end{proposition}
\begin{proof}[Sketch of proof]
Denote $\{\rho_n\}_{n\geq1}$ the Markov trace on $\{\hY_{d,n}\}_{n\geq1}$ defined by the right hand side of (\ref{compare1}). To prove the proposition, we need to check that Condition (\ref{cond-x}) is satisfied. 

For $n=1$, it is a straightforward verification. For $n>1$, one may check the equivalent condition $\rho_{n}(\tX^a_nt_n^b\,h)=|D|v^{-1}(1-u^2)x_{a,b}\,\rho_{n-1}(h)$. We note that it is enough to take $h=E_{\chi}X^{\lambda}g_{w^{-1}}$, where $\chi\in\text{Irr}(\mathcal{T}_{d,n})$ is such that $\Comp(\chi)=\mu$ where $\overline{\mu}=\mu^S$ for some $S\subset D$, and such that $w(\chi)=\chi$ (otherwise the condition is $0=0$). Then the condition can be checked by a straightforward calculation of both sides. One may use: an explicit description of the embedding $\hY_{d,n}\subset\hY_{d,n+1}$ on the matrix algebras side (see \cite[Section 3.4]{JaPo}); and the fact that $\Psi_{d,n}(\tX_j)=\sum_{\chi}\bun_{\chi,\chi}\overline{\tX}_{\pi_{\chi}^{-1}(j)}$ for $j=1,\dots,n$ (induction on $j$).
\end{proof}
We note that Formula (\ref{compare1}) is a triangular change of basis, with inverse
\[\{\rho_{d,n}^{S,\btau}\}_{n\geq1}=\bigl(v^{-1}(1-u^2)\bigr)^{1-\vert S\vert}\sum_{D\subseteq S}(-1)^{|S|- |D|}\vert D\vert\,\{\trho_{d,n}^{D,\bx}\}_{n\geq1}\,.\]
In particular, restricting to $Y_{d,n}$, we can forget the parameters $\bx$ on one hand, and the choice $\btau$ on the other. The proposition expresses the Juyumaya--Lambropoulou invariant associated to $D\subset\{1,\dots,d\}$ in terms of our basic invariants associated to $S\subset\{1,\dots,d\}$. The formulas are (with notations as in Remark \ref{rem-compare1} and coefficients $\alpha_{S}$ as in (\ref{compare1})):
\begin{equation}\label{compare2}
\Phi_{d,D}(q,z)(L)=\sum_{S\subseteq D}\alpha_{S}\, P_L^{d,S}(u,v,1)\ \ \ \ \text{and}\ \ \ \ \ \Gamma_{d,D}(q,z)(L)=\sum_{S\subseteq D}\alpha_{S}\, P_L^{d,S}(u,v,q^{-1})\,.
\end{equation}

\section{Properties of invariants}

As consequences of the construction using the isomorphism theorem, we prove several properties of the constructed invariants, focusing essentially on the non-framed links. We emphasize that these properties are valid for all non-framed links (classical and solid torus). 

\subsection{Comparison of invariants with different $d$ for non-framed links}\label{subsec-d}

Let $d>0$ and $S$ a non-empty subset of $\{1,\dots,d\}$. We denote $d':=|S|$. Let $\btau$ be any choice of $d'$ Markov traces on $\{\hH_n\}_{n\geq1}$. The following result says that, for non-framed links, it is enough to consider the situation $S=\{1,\dots,d\}$ for each $d>0$.
\begin{proposition}\label{prop1}
For any non-framed link $L$, we have $P^{d,S,\btau}_L=P^{d',\{1,\dots,d'\},\btau}_L$\,.
\end{proposition}
 In particular, if $|S|=1$, the proposition asserts that $P^{d,S,\tau}_L(u,v,\gamma)=P^{\tau}_L(u,v)$, where $P^{\tau}_L(u,v)$ is the invariant of the non-framed link obtained from $\hH_n$.
\begin{proof} Let $\{\xi^{(d)}_1,\dots,\xi^{(d)}_{d}\}$ denote the $d$-th roots of unity.

Let $\mu\mmodels_d n$ be such that $\overline{\mu}=\mu^S$, that is, such that $\mu_a\neq 0$ if and only if $a\in S$. To $\mu$, we associate the composition $\mu'=(\mu_{i_1},\dots,\mu_{i_{d'}})\mmodels_{d'}n$, where $\mu_{i_1},\dots,\mu_{i_{d'}}$ are the non-zero parts of $\mu$ and $i_1<\dots<i_{d'}$. We have $S^{\mu}=S^{\mu'}$ and, in turn, $\hH^{\mu}=\hH^{\mu'}$ and $m_{\mu}=m_{\mu'}$.

Let $\chi\in\Irr(\mathcal{T}_{d,n})$ with $\Comp(\chi)=\mu$. For every $j=1,\dots,n$, by hypothesis on $\mu$ and $\chi$, there exists $a\in\{1,\dots,d'\}$ such that $\chi(t_j)=\xi^{(d)}_{i_a}$. Then we set $\chi'(t_j)=\xi^{(d')}_a$. This defines a bijection between the characters $\chi\in\Irr(\mathcal{T}_{d,n})$ with $\Comp(\chi)=\mu$ and the characters $\chi'\in\Irr(\mathcal{T}_{d',n})$ with $\Comp(\chi')=\mu'$. 
This bijection allows to identify the spaces $\Mat_{m_{\mu}}(\hH^{\mu})$ and $\Mat_{m_{\mu'}}(\hH^{\mu'})$. We have moreover $\pi_{\chi}=\pi_{\chi'}$.

Recall the formulas (\ref{form-X})-(\ref{form-g}) giving the images of the generators $g_1,\dots,g_{n-1},X_1$ under $\Psi_{d,n}$. It is then immediate to see that $\Psi_{d,n}(x)$  and $\Psi_{d',n}(x)$ coincide in the summand $\Mat_{m_{\mu}}(\hH^{\mu})=\Mat_{m_{\mu'}}(\hH^{\mu'})$, for any $x$ in the subalgebra of $\hY_{d,n}$ generated by $g_1,\dots,g_{n-1},X_1$. 
\end{proof}

\begin{remark}\label{rem-compare2} Here we restrict to a classical non-framed link $L$. With the comparison formulas (\ref{compare2}) of the preceding section, a straightforward consequence of Proposition \ref{prop1} is the corresponding result for the Juyumaya--Lambropoulou invariants. Namely, we have 
\[\Phi_{d,D}(q,z)(L)=\Phi_{d',\{1,\dots,d'\}}(q,z)(L)\ \ \ \text{and}\ \ \ \Gamma_{d,D}(q,z)(L)=\Gamma_{d',\{1,\dots,d'\}}(q,z)(L)\,,\] 
where $d'=|D|$. In this case, this was proved in \cite[Proposition 4.6]{ChJKL} by a different approach. In particular if $|D|=1$, we recover the HOMFLYPT polynomial.
\end{remark}

\subsection{Links with a fixed number of connected components}

Let $\mu\mmodels_d n$ and $\lambda\mmodels_kn$ for some $d,k>0$. We will say that $\lambda$ is a \emph{refinement} of $\mu$ if $\{1,\dots,k\}$ can be partitioned into $d$ disjoint subsets (possibly empty): $\{1,\dots,k\}=I_1\sqcup\dots\sqcup I_d$, such that $\mu_a=\sum_{i\in I_a}\lambda_i$ for all $a=1,\dots,d$. As examples, every composition $\lambda\mmodels n$ is a refinement of the composition $(n)$, while the composition $(1,\dots,1)\mmodels n$ is a refinement of every composition $\mu\mmodels n$.

For a permutation $\pi\in S_n$, we denote $\text{cyc}(\pi)$ the collection of lengths of the cycles of $\pi$ and we consider it as a composition of $n$ (the order is not relevant here). Then, it is immediate that a permutation $\pi\in S_n$ is conjugate to an element of $S^{\mu}$ if and only if $\text{cyc}(\pi)$ is a refinement of $\mu$.

For a $\Z/d\Z$-framed affine braid $\beta\in \Z/d\Z\wr \hB_n$, we define its \emph{underlying permutation} $p_{\beta}$ as the image of $\beta$ by the natural group homomorphism from $\Z/d\Z\wr \hB_n$ to $S_n$ (defined by $t_j\mapsto 1$, $\sigma_0\mapsto 1$ and $\sigma_i\mapsto s_i$). Note that $\beta t_j=t_{p_{\beta}(j)}\beta$, for $j=1,\dots,n$.

\begin{proposition}\label{prop2}
Let $\beta\in \Z/d\Z\wr \hB_n$. In $\Psi_{d,n}\bigl(\delta_{d,n}(\beta)\bigr)$, the matrix corresponding to $\mu\mmodels_d n$ has all its diagonal elements equal to $0$ if $\text{cyc}(\pi)$ is not a refinement of $\mu$.
\end{proposition}
\begin{proof}
Let $x:=\delta_{d,n}(\beta)\in\hY_{d,n}$. As $\delta_{d,n}$ is a group homomorphism, we have $xt_j=t_{p_{\beta}(j)}x$.
In $\Psi_{d,n}(x)$, in the matrix corresponding to $\mu\mmodels_d n$, the coefficient on the diagonal in position $\chi$ is $\Psi_{d,n}(E_{\chi}xE_{\chi})$. And we have $E_{\chi}xE_{\chi}=E_{\chi}E_{p_{\beta}(\chi)}x$. This is equal to 0 if $p_{\beta}(\chi)\neq\chi$. Now $p_{\beta}(\chi)=\chi$ if and only if $\pi_{\chi}^{-1}p_{\beta}\pi_{\chi}\in S^{\mu}$, and this is impossible if $\text{cyc}(p_{\beta})$ is not a refinement of $\mu$. 
\end{proof}

Now we will combine this general result on the isomorphism $\Psi_{d,n}$ with two elementary facts. First, if a composition $\mu$ has strictly more non-zero parts than another composition $\lambda$, then $\lambda$ can not be a refinement of $\mu$. Second, the closure of a ($\Z/d\Z$-framed, affine) braid $\beta$ is a ($\Z/d\Z$-framed) link with $N$ connected components if and only if $\text{cyc}(p_{\beta})$ has exactly $N$ non-zero parts. Thus we have obtained the following result.

\begin{corollary}\label{cor-fin1}
Let $L$ be a $\Z/d \Z$-framed link with $N$ connected components. We have $P^{d,S,\btau}_L=0$ if $|S|>N$.
\end{corollary}

Finally, combining this corollary with Proposition \ref{prop1} for a non-framed link $L$, we conclude our study by determining which basic invariants it is enough to consider given the number of connected components of $L$. 
\begin{corollary}\label{cor-fin2}
Let $L$ be a non-framed link with $N$ connected components. Every invariant for $L$ obtained here from (affine) Yokonuma--Hecke algebras is a combination of the following basic invariants:
\begin{equation}\label{list-inv}
P^{1,\{1\},\btau}_L\,,\ P^{2,\{1,2\},\btau}_L\,,\ \dots\ ,\ P^{N,\{1,\dots,N\},\btau}_L\ .
\end{equation}
In particular, for a non-framed knot ($N=1$), it is enough to consider the algebras $\hH_n=\hY_{1,n}$.
\end{corollary}

For a classical non-framed link, we can forget the parameters $\btau$, and it is enough to calculate $N$ distinct invariants, the first one being the HOMFLYPT polynomial.

\begin{remark}\label{rem-compare3}
Here we restrict to a classical $\Z/d \Z$-framed link $L$ with $N$ connected components. In this particular case, we give the translation of Corollary \ref{cor-fin1} in terms of Juyumaya--Lambropoulou invariants, following Formula \ref{compare2} (we write the formulas for the invariants $\Phi_{d,D}$; the same formulas hold for $\Gamma_{d,D}$). A straightforward analysis leads to
\[\Phi_{d,D}(q,z)(L)=\frac{N}{|D|}\sum_{D'\subset D,\ |D'|=N}\Phi_{d,D'}(q,\frac{|D|}{N}z)(L)\,,\ \ \ \ \ \ \ \text{if $|D|>N$\,.}\]
Note that the rescaling of the variable $z$ comes from the fact that the expressions relating $(u,v)$ with $(q,z)$ depend on $|D|$.

If moreover $L$ is a classical non-framed link with $N$ connected components, with Remark \ref{rem-compare2}, it is enough to consider $D=\{1,\dots,d\}$, and we obtain
\begin{equation}\label{form-Phi}
\Phi_{d,\{1,\dots,d\}}(q,z)(L)=\frac{N}{d}\left(\begin{array}{c}d\\ N\end{array}\right)\Phi_{N,\{1,\dots,N\}}(q,\frac{d}{N}z)(L)\,,\ \ \ \ \ \ \ \text{if $d>N$\,.}
\end{equation}
This formula is the generalisation, for $N>1$, of \cite[Theorem 5.8]{ChJKL}. As a consequence, the analogue of Corollary \ref{cor-fin2} holds as well for these invariants: it is enough to consider $\Phi_{1,\{1\}}$, $\dots$, $\Phi_{N,\{1,\dots,N\}}$. This generalises \cite[Theorem 7.1]{ChJKL} for $N>2$, obtained by a different method.
\end{remark}

\end{document}